\documentclass[10pt,twoside]{article}
\usepackage{mathrsfs}
\usepackage{amssymb}
\usepackage{amsmath,color}
\usepackage[all]{xy}
\usepackage{amsthm}

\usepackage{url}
\numberwithin{equation}{section}

\newtheorem{theorem}{Theorem}[section]
\newtheorem{proposition}[theorem]{Proposition}
\newtheorem{lemma}[theorem]{Lemma}

\newtheorem{corollary}[theorem]{Corollary}
\newtheorem{question}[theorem]{Question}
\newtheorem{theorem*}{Theorem}

\theoremstyle{definition}
\newtheorem{definition}[theorem]{Definition}
\newtheorem{example}[theorem]{Example}

\theoremstyle{remark}
\newtheorem{remark}[theorem]{Remark}

\usepackage{paralist}

\newcommand{\Mod}{\operatorname{Mod}}
\newcommand{\Hom}{\operatorname{Hom}}
\newcommand{\Ext}{\operatorname{Ext}}

\newcommand{\ra}{\rightarrow}

\def\Im{\mathop{\rm Im}\nolimits}
\def\Ker{\mathop{\rm Ker}\nolimits}

\def\mod{\mathop{\rm mod}\nolimits}
\def\Mod{\mathop{\rm Mod}\nolimits}
\def\Hom{\mathop{\rm Hom}\nolimits}

\def\Ext{\mathop{\rm Ext}\nolimits}

\def\inf{\mathop{\rm inf}\nolimits}

\def\dim{\mathop{\rm dim}\nolimits}

\def\Im{\mathop{\rm Im}\nolimits}
\def\Ker{\mathop{\rm Ker}\nolimits}

\def\mod{\mathop{\rm mod}\nolimits}
\def\Mod{\mathop{\rm Mod}\nolimits}

\def\inf{\mathop{\rm inf}\nolimits}

\def\dim{\mathop{\rm dim}\nolimits}

\def\Hom{\mathop{\rm Hom}\nolimits}
\def\Ext{\mathop{\rm Ext}\nolimits}

\def\lim{\mathop{\underrightarrow{\rm lim}}\nolimits}

\def\spc{\mathop{\rm SPC}\nolimits}
\def\pc{\mathop{\rm PC}\nolimits}
\def\spe{\mathop{\rm SPE}\nolimits}

\title{ \bf Special Precovered Categories of Gorenstein Categories\thanks{2010 Mathematics Subject Classification: 18G25, 18E10.}
\thanks{Keywords: Gorenstein categories, Right 1-orthogonal categories, Special precovers,
Special precovered categories, Projectively resolving, Injectively coresolving.
 }}
\vspace{0.2cm}

\author{Tiwei Zhao\thanks{E-mail address:  tiweizhao@hotmail.com} \ and Zhaoyong Huang\thanks{E-mail address: huangzy@nju.edu.cn
} \\
{\it \footnotesize  Department of Mathematics, Nanjing University, Nanjing 210093, Jiangsu Province, P. R. China}}
\date{ }
\begin{document}

\baselineskip=16pt
\maketitle

\begin{abstract}
Let $\mathscr{A}$ be an abelian category and $\mathscr{P}(\mathscr{A})$ the subcategory of $\mathscr{A}$
consisting of projective objects. Let $\mathscr{C}$ be a full, additive and self-orthogonal subcategory
of $\mathscr{A}$ with $\mathscr{P}(\mathscr{A})$ a generator, and let $\mathcal{G}(\mathscr{C})$
be the Gorenstein subcategory of $\mathscr{A}$. Then the right 1-orthogonal category ${\mathcal{G}(\mathscr{C})^{\bot_1}}$ of
$\mathcal{G}(\mathscr{C})$ is both projectively
resolving and injectively coresolving in $\mathscr{A}$. We also get that the subcategory
$\spc(\mathcal{G}(\mathscr{C}))$ of $\mathscr{A}$ consisting of objects admitting special $\mathcal{G}(\mathscr{C})$-precovers
is closed under extensions and $\mathscr{C}$-stable direct summands (*). Furthermore, if $\mathscr{C}$ is a generator for
$\mathcal{G}(\mathscr{C})^{\perp_1}$, then we have that $\spc(\mathcal{G}(\mathscr{C}))$
is the minimal subcategory of $\mathscr{A}$ containing $\mathcal{G}(\mathscr{C})^{\perp_1}\cup \mathcal{G}(\mathscr{C})$
with respect to the property (*), and that $\spc(\mathcal{G}(\mathscr{C}))$ is $\mathscr{C}$-resolving in $\mathscr{A}$
with a $\mathscr{C}$-proper generator $\mathscr{C}$.
\end{abstract}

\pagestyle{myheadings}
\markboth{\rightline {\scriptsize   T. W. Zhao and Z. Y. Huang}}
         {\leftline{\scriptsize Special Precovered Categories of Gorenstein Categories}}


\section{Introduction} 

As a generalization of finitely generated projective modules, Auslander and Bridger introduced in \cite{AB}
the notion of finitely generated modules of Gorenstein dimension zero over commutative Noetherian rings. Then
Enochs and Jenda generalized it in \cite{EJ95} to arbitrary modules over a general ring and introduced the notion
of Gorenstein projective modules and its dual (that is, the notion of Gorenstein injective modules).
Let $\mathscr{A}$ be an abelian category and $\mathscr{C}$ an
additive and full subcategory of $\mathscr{A}$. Recently Sather-Wagstaff, Sharif
and White introduced in \cite{SSW08} the notion of the Gorenstein subcategory
$\mathcal{G}(\mathscr{C})$ of $\mathscr{A}$, which is a common generalization of
the notions of modules of Gorenstein dimension zero \cite{AB}, Gorenstein projective
modules, Gorenstein injective modules \cite{EJ95}, $V$-Gorenstein
projective modules and $V$-Gorenstein injective modules \cite{EJL}, and so on.

Let $R$ be an associative ring with identity, and let $\Mod R$ be the category of left $R$-modules
and $\mathcal{G}(\mathscr{P}(\Mod R)$ the subcategory of $\Mod R$ consisting of Gorenstein projective modules.
Let $\pc(\mathcal{G}(\mathscr{P}(\Mod R))$ and $\spc(\mathcal{G}(\mathscr{P}(\Mod R))$ be the subcategories of $\Mod R$ consisting of modules
admitting a $\mathcal{G}(\mathscr{P}(\Mod R))$-precover and admitting a special $\mathcal{G}(\mathscr{P}(\Mod R))$-precover respectively.
The following question in relative homological algebra remains still open: does $\pc(\mathcal{G}(\mathscr{P}(\Mod R))=\Mod R$
always hold true? Several authors have gave some partially positive answers to this question, see \cite{ADH, BGM, CFH11, WL}.
Note that in these references, $\pc(\mathcal{G}(\mathscr{P}(\Mod R))=\spc(\mathcal{G}(\mathscr{P}(\Mod R))$, see Example \ref{4.8} below for details.
In particular, any module in $\Mod R$ with finite Gorenstein projective dimension admits a $\mathcal{G}(\mathscr{P}(\Mod R))$-precover
which is also special (\cite{Ho04}). In fact, it is unknown whether $\pc(\mathcal{G}(\mathscr{P}(\Mod R))=\spc(\mathcal{G}(\mathscr{P}(\Mod R))$
always holds true. Based on the above, it is necessary to study the properties of these two subcategories.

Let $\mathscr{A}$ be an abelian category and $\mathscr{C}$ an
additive and full subcategory of $\mathscr{A}$. We use $\spc(\mathcal{G}(\mathscr{C}))$ to denote the subcategory of
$\mathscr{A}$ consisting of objects admitting special $\mathcal{G}(\mathscr{C})$-precovers. The aim of this paper
is to investigate the structure of $\spc(\mathcal{G}(\mathscr{C}))$ in terms of the properties of the right 1-orthogonal
category $\mathcal{G}(\mathscr{C})^{\perp_1}$ of $\mathcal{G}(\mathscr{C})$.
This paper is organized as follows.

In Section 2, we give some terminology and some preliminary results.

Assume that $\mathscr{C}$ is self-orthogonal and the subcategory of $\mathscr{A}$ consisting of projective objects
is a generator for $\mathscr{C}$. In Section 3, we prove that $\mathcal{G}(\mathscr{C})^{\perp_1}$
is both projectively resolving and injectively coresolving in $\mathscr{A}$.
We also characterize when all objects in $\mathscr{A}$ are in $\mathcal{G}(\mathscr{C})^{\perp_1}$.

In Section 4, we prove that $\spc(\mathcal{G}(\mathscr{C}))$
is closed under extensions and $\mathscr{C}$-stable direct summands (*). Furthermore,
if $\mathscr{C}$ is a generator for $\mathcal{G}(\mathscr{C})^{\perp_1}$, then we get the following two results:
(1) $\spc(\mathcal{G}(\mathscr{C}))$ is the minimal subcategory of $\mathscr{A}$ containing
$\mathcal{G}(\mathscr{C})^{\perp_1}\cup \mathcal{G}(\mathscr{C})$ with respect to the property (*);
and (2) $\spc(\mathcal{G}(\mathscr{C}))$ is $\mathscr{C}$-resolving in $\mathscr{A}$ with a $\mathscr{C}$-proper generator $\mathscr{C}$.

\section{Preliminaries}

Throughout this paper, $\mathscr{A}$ is an abelian category and all subcategories of $\mathscr{A}$ are full, additive and
closed under isomorphisms. We use $\mathscr{P}(\mathscr{A})$
(resp. $\mathscr{I}(\mathscr{A})$) to denote the subcategory of $\mathscr{A}$ consisting of projective
(resp. injective) objects. For a subcategory $\mathscr{C}$ of $\mathscr{A}$ and an object $A$ in $\mathscr{A}$,
the {\it $\mathscr{C}$-dimension} $\mathscr{C}{\text-}\dim A$ of $A$
is defined as $\inf\{n\geq 0\mid$ there exists an exact sequence
$$0\to C_{n} \to \cdots \to C_{1} \to C_{0} \to A \to 0$$ in $\mathscr{A}$
with all $C_i$ in $\mathscr{C}\}$. Set $\mathscr{C}{\text-}\dim A=\infty$ if no such integer exists (cf. \cite{Hu14}).
For a non-negative integer $n$, we use $\mathscr{C}^{\leq n}$ (resp. $\mathscr{C}^{<\infty}$) to denote the subcategory
of $\mathscr{A}$ consisting of objects with $\mathscr{C}$-dimension at most $n$ (resp. finite $\mathscr{C}$-dimension).

Let $\mathscr{X}$ be a subcategory of $\mathscr{A}$. Recall that a sequence in
$\mathscr{A}$ is called {\it $\Hom_{\mathscr{A}}(\mathscr{X},-)$-exact} if it is exact after
applying the functor $\Hom_{\mathscr{A}}(X,-)$ for any $X\in\mathscr{X}$. Dually,
the notion of a {\it $\Hom_{\mathscr{A}}(-, \mathscr{X})$-exact sequence} is defined.
Set
$$\mathscr{X}^\perp:=\{M\mid \Ext^{\geq 1}_\mathscr{A}(X,M)=0\ {\rm for\ any}\ X\in \mathscr{X}\},$$
$${^\perp\mathscr{X}}:=\{M\mid \Ext^{\geq 1}_\mathscr{A}(M,X)=0\ {\rm for\ any}\ X\in \mathscr{X}\},$$
and
$${\mathscr{X}^{\perp_1}}:=\{M\mid \Ext^{1}_\mathscr{A}(X,M)=0\ {\rm for\ any}\ X\in \mathscr{X}\},$$
$${{^{\perp_1}\mathscr{X}}}:=\{M\mid \Ext^{1}_\mathscr{A}(M,X)=0\ {\rm for\ any}\ X\in \mathscr{X}\}.$$
We call ${\mathscr{X}^{\perp_1}}$ (resp. ${{^{\perp_1}\mathscr{X}}}$) the {\it right (resp. left) 1-orthogonal category of $\mathscr{X}$}.
Let $\mathscr{X}$ and $\mathscr{Y}$ be subcategories of $\mathscr{A}$.
We write $\mathscr{X}\perp \mathscr{Y}$ if $\Ext^{\geq 1}_\mathscr{A}(X,Y)=0$ for any $X\in \mathscr{X}$
and $Y\in \mathscr{Y}$.

\vspace{0.2cm}

\begin{definition} \label{2.1}
{ (cf. \cite{E}) Let $\mathscr{X}\subseteq\mathscr{Y}$ be
subcategories of $\mathscr{A}$. The morphism $f: X\to Y$ in
$\mathscr{A}$ with $X\in\mathscr{X}$ and $Y\in \mathscr{Y}$ is called an
{\it $\mathscr{X}$-precover} of $Y$
if $\Hom_{\mathscr{A}}(X',f)$ is epic for any $X'\in\mathscr{X}$.
An $\mathscr{X}$-precover $f: X\to Y$ is called {\it special} if $f$ is
epic and $\Ker f\in \mathscr{X}^{\bot_1}$. $\mathscr{X}$ is called {\it special precovering}
in $\mathscr{Y}$ if any object in $\mathscr{Y}$ admits a special $\mathscr{X}$-precover.
Dually, the notions of a {\it (special) $\mathscr{X}$-(pre)envelope} and a
{\it special preenveloping subcategory} are defined.}
\end{definition}

\begin{definition} \label{2.2}
{ (cf. \cite{Ho04}) A subcategory of $\mathscr{A}$ is called {\it projectively resolving}
if it contains $\mathscr{P}(\mathscr{A})$ and is closed under extensions and under kernels of epimorphisms.
Dually, the notion of {\it injectively coresolving subcategories} is defined.}
\end{definition}


From now on, assume that $\mathscr{C}$ is a given subcategory of $\mathscr{A}$.

\begin{definition} \label{2.3}
{ (cf. \cite{SSW08}) The {\it Gorenstein subcategory}
$\mathcal{G}(\mathscr{C})$ of $\mathscr{A}$ is defined as
$\mathcal{G}(\mathscr{C})=\{M$ is an object in $\mathscr{A}\mid$
there exists an exact sequence: $$\cdots \to C_1 \to C_0 \to C^0 \to
C^1 \to \cdots \eqno{(2.1)}$$ in $\mathscr{C}$, which is both
$\Hom_{\mathscr{A}}(\mathscr{C},-)$-exact and
$\Hom_{\mathscr{A}}(-,\mathscr{C})$-exact, such that $M\cong
\Im(C_0\to C^0)\}$; in this case, $(2.1)$ is called a {\it complete
$\mathscr{C}$-resolution} of $M$.}
\end{definition}

In what follows, $R$ is an associative ring with identity, $\Mod R$ is
the category of left $R$-modules and $\mod R$ is
the category of finitely generated left $R$-modules.

\begin{remark} \label{2.4}{
\begin{enumerate}
\item[]
\item[(1)]
Let $R$ be a left and right Noetherian ring. Then
$\mathcal{G}(\mathscr{P}(\mod R))$ coincides with the subcategory of $\mod
R$ consisting of modules with Gorenstein dimension zero (\cite{AB}).

\item[(2)] $\mathcal{G}(\mathscr{P}(\Mod R))$
(resp. $\mathcal{G}(\mathscr{I}(\Mod R))$)
coincides with the subcategory of $\Mod R$ consisting of Gorenstein
projective (resp. injective) modules (\cite{EJ95}).

\item[(3)] Let $R$ be a left Noetherian ring, $S$ a right Noetherian ring
and $_RV_S$ a dualizing bimodule. Put $\mathscr{W}=\{V\bigotimes_SP\mid P\in
\mathscr{P}(\Mod S)\}$ and $\mathscr{U}=\{\Hom_S(V,E)\mid E\in
\mathscr{I}(\Mod S^{op})\}$. Then $\mathcal{G}(\mathscr{W})$
(resp. $\mathcal{G}(\mathscr{U})$) coincides
with the subcategory of $\Mod R$ consisting of $V$-Gorenstein
projective (resp. injective) modules (\cite{EJL}).
\end{enumerate}}
\end{remark}

\begin{definition}\label{2.5} {
(cf. \cite{SSW08}) Let $\mathscr{X}\subseteq\mathscr{T}$ be
subcategories of $\mathscr{A}$. Then $\mathscr{X}$ is called a {\it generator} (resp.
{\it cogenerator}) for $\mathscr{T}$ if for any $T\in
\mathscr{T}$, there exists an exact sequence $0\to T' \to X \to
T \to 0$ (resp. $0\to T \to X \to T' \to 0$) in $\mathscr{T}$ with $X\in\mathscr{X}$;
and $\mathscr{X}$ is called a {\it projective generator} (resp.
{\it an injective cogenerator})  for $\mathscr{T}$  if $\mathscr{X}$ is a generator (resp.
cogenerator) for $\mathscr{T}$ and $\mathscr{X}\perp \mathscr{T}$ (resp. $\mathscr{T}\perp \mathscr{X}$).}
\end{definition}

We have the following easy observation.

\begin{lemma}\label{2.6}
Assume that $\mathscr{C}\perp \mathscr{C}$ and $\mathscr{P}(\mathscr{A})$ is a generator for $\mathscr{C}$.
Then for any $G\in \mathcal{G}(\mathscr{C})$, there exists a $\Hom_{\mathscr{A}}(\mathscr{C},-)$-exact and
$\Hom_{\mathscr{A}}(-,\mathscr{C})$-exact exact sequence
$$0\ra G'\ra P\ra G\ra 0$$ in $\mathscr{A}$ with $P\in \mathscr{P}(\mathscr{A})$ and $G'\in \mathcal{G}(\mathscr{C})$.
\end{lemma}

\begin{proof}
Let $G\in \mathcal{G}(\mathscr{C})$. Then there exists a $\Hom_{\mathscr{A}}(\mathscr{C},-)$-exact and
$\Hom_{\mathscr{A}}(-,\mathscr{C})$-exact exact sequence
$$0\ra G_1\ra C_0\ra G\ra 0$$ in $\mathscr{A}$ with $C_0\in \mathscr{C}$ and $G_1\in \mathcal{G}(\mathscr{C})$.
Because $\mathscr{P}(\mathscr{A})$ is a generator for $\mathscr{C}$ by assumption, there exists an exact sequence
$$0\ra C'\ra P\ra C_0\ra 0$$ in $\mathscr{A}$ with $P\in \mathscr{P}(\mathscr{A})$ and $C'\in \mathscr{C}$.
Consider the following pullback diagram
$$
\xymatrix@R=20pt@C=20pt{ & 0 \ar@{-->}[d] & 0 \ar[d]& &\\
 & C'\ar@{==}[r] \ar@{-->}[d] &C' \ar[d]& &\\
0 \ar@{-->}[r] & G' \ar@{-->}[d] \ar@{-->}[r] & P \ar[d] \ar@{-->}[r] &G \ar@{==}[d] \ar@{-->}[r] & 0\\
0 \ar[r] & G_1 \ar[r] \ar[d] & C_0 \ar[r] \ar[d] & G \ar[r] & 0 &\\
 & 0 & 0. & & }
$$
By \cite[Lemma 2.5]{Hu}, the middle row is both $\Hom_{\mathscr{A}}(\mathscr{C},-)$-exact and $\Hom_{\mathscr{A}}(-,\mathscr{C})$-exact,
and hence $G'\in \mathcal{G}(\mathscr{C})$ by \cite[Proposition 4.7]{Hu}, that is, the middle row is the desired sequence.
\end{proof}

The following result is useful in the sequel.

\begin{proposition}\label{2.7}
Assume that $\mathscr{C}\perp \mathscr{C}$ and $\mathscr{P}(\mathscr{A})$ is a generator for $\mathscr{C}$.  Then
\begin{enumerate}
\item[(1)] $\mathcal{G}(\mathscr{C})^{\perp_1}=\mathcal{G}(\mathscr{C})^{\perp}$.
\item[(2)] $\mathcal{G}(\mathscr{C})\subseteq {^{\bot}\mathscr{C}\cap \mathscr{C}^{\bot}}$.
\end{enumerate}
\end{proposition}

\begin{proof}
(1) It suffices to prove that $\mathcal{G}(\mathscr{C})^{\perp_1}\subseteq\mathcal{G}(\mathscr{C})^{\perp}$.
Let $M\in \mathcal{G}(\mathscr{C})^{\perp_1}$ and $G\in \mathcal{G}(\mathscr{C})$. By Lemma \ref{2.6}, we have an exact sequence
$$0\ra G'\ra P\ra G\ra 0$$ in $\mathscr{A}$ with $P\in \mathscr{P}(\mathscr{A})$ and $G'\in \mathcal{G}(\mathscr{C})$.
It induces $\Ext^2_\mathscr{A}(G,M)\cong\Ext^1_\mathscr{A}(G',M)=0$, and hence $\Ext^2_\mathscr{A}(G',M)=0$ and
$\Ext^3_\mathscr{A}(G,M)\cong\Ext^2_\mathscr{A}(G',M)=0$. Repeating this process, we get $\Ext^{\geq 1}_\mathscr{A}(G,M)=0$.

(2) See \cite[Lemma 5.7]{Hu}.
\end{proof}

We remark that if $\mathscr{A}$ has enough projective objects, and if $\mathscr{P}(\mathscr{A})\subseteq \mathscr{C}$
and $\mathscr{C}$ is closed under kernels of epimorphisms,
then $\mathscr{P}(\mathscr{A})$ is a generator for $\mathscr{C}$.

\section{The right 1-orthogonal category of $\mathcal{G}(\mathscr{C})$}

In the rest of this paper, assume that the subcategory $\mathscr{C}$ is self-orthogonal (that is, $\mathscr{C}\perp\mathscr{C}$)
and $\mathscr{P}(\mathscr{A})$ is a generator for $\mathscr{C}$. In this section, we mainly investigate
the homological properties of $\mathcal{G}(\mathscr{C})^{\perp_1}$. We begin with some examples of $\mathcal{G}(\mathscr{C})^{\perp_1}$.

\begin{example}\label{3.1}
{ \begin{enumerate}
\item[]
\item[(1)] By Proposition \ref{2.7} and \cite[Theorem 5.8]{Hu}, we have $\mathscr{P}(\mathscr{A})
\subseteq\mathscr{C}\subseteq{\mathscr{C}}^{<\infty}\subseteq \mathcal{G}(\mathscr{C})^{\perp_1}$.
\item[(2)] $\mathscr{P}(\mathscr{A})^{<\infty}\cup\mathscr{I}(\mathscr{A})^{<\infty}\subseteq \mathcal{G}(\mathscr{C})^{\perp_1}$.
\item[(3)] If the global dimension of $R$ is finite, then $\mathcal{G}(\mathscr{P}(\Mod R))^{\perp_1}=\Mod R$.
\item[(4)] By \cite[Theorem 11.5.1]{EJ} and \cite[Theorem 31.9]{AF}, we have that $R$
is quasi-Frobenius if and only if $\mathcal{G}(\mathscr{P}(\Mod R))^{\perp_1}=\mathscr{I}(\Mod R)$, and if and only if
$\mathcal{G}(\mathscr{P}(\Mod R))^{\perp_1}=\mathscr{P}(\Mod R)=\mathscr{I}(\Mod R)$.
\end{enumerate}}
\end{example}

For a non-negative integer $n$, recall that a left and right noetherian ring $R$ is called {\it $n$-Gorenstein}
if the left and right self-injective dimensions of $R$ are at most $n$. The following result is a generalization
of Example \ref{3.1}(4).

\begin{example}\label{3.2}
If $R$ is $n$-Gorenstein, then
$$\mathcal{G}(\mathscr{P}(\Mod R))^{\perp_1}=\mathscr{P}(\Mod R)^{\leq n}=\mathscr{P}(\Mod R)^{<\infty}
=\mathscr{I}(\Mod R)^{\leq n}=\mathscr{I}(\Mod R)^{<\infty}.$$
\end{example}

\begin{proof}
By \cite[Theorem 2]{Iw80} and Example \ref{3.1}(2), we have $$\mathscr{P}(\Mod R)^{\leq n}=\mathscr{P}(\Mod R)^{<\infty}=
\mathscr{I}(\Mod R)^{\leq n}=\mathscr{I}(\Mod R)^{<\infty}\subseteq \mathcal{G}(\mathscr{P}(\Mod R))^{\perp_1}.$$

Now let $M\in \mathcal{G}(\mathscr{P}(\Mod R))^{\perp_1}$ and $N\in \Mod R$. Since $R$ is $n$-Gorenstein, there exists an exact sequence
$$0\ra G_n\ra P_{n-1}\ra \cdots \ra  P_0\ra M\ra 0$$ in $\Mod R$ with all $P_i$ in $\mathscr{P}(\Mod R)$ and $G_n\in
\mathcal{G}(\mathscr{P}(\Mod R))$ by \cite[Theorem 11.5.1]{EJ}. Then we have $\Ext^{n+1}_R(N,M)\cong \Ext^1_R(G_n,M)=0$
and $M\in\mathscr{I}(\Mod R)^{\leq n}$, and thus $\mathcal{G}(\mathscr{P}(\Mod R))^{\perp_1}\subseteq\mathscr{I}(\Mod R)^{\leq n}$.
\end{proof}

The following result shows that $\mathcal{G}(\mathscr{C})^{\perp_1}$ behaves well.

\begin{theorem}\label{3.3}
\begin{enumerate}
\item[]
\item[(1)] $\mathcal{G}(\mathscr{C})^{\perp_1}$ is closed under direct products, direct summands and extensions.
\item[(2)] $\mathcal{G}(\mathscr{C})^{\perp_1}$ is projectively resolving in $\mathscr{A}$.
\item[(3)] $\mathcal{G}(\mathscr{C})^{\perp_1}$ is injectively coresolving in $\mathscr{A}$.
\end{enumerate}
\end{theorem}

\begin{proof}
(1) It is trivial.

(2) By Example \ref{3.1}(1), $\mathscr{P}(\mathscr{A})\subseteq \mathcal{G}(\mathscr{C})^{\perp_1}$.
Let $G\in\mathcal{G}(\mathscr{C})$ and
$$0\ra L\ra M\ra N\ra 0$$
be an exact sequence in $\mathscr{A}$ with $M,N\in\mathcal{G}(\mathscr{C})^{\perp_1}$. By Proposition \ref{2.7}(1), we have
$\Ext^{\geq 1}_{\mathscr{A}}(G,M)=0=\Ext^{\geq 1}_\mathscr{A}(G,N)$. Then $\Ext^{\geq 2}_\mathscr{A}(G,L)=0$.
Because $G\in\mathcal{G}(\mathscr{C})$, we have an exact sequence
$$0\ra G\ra C^0\ra G^1\ra 0$$ in $\mathscr{A}$ with $C^0\in\mathscr{C}$ and $G^1\in\mathcal{G}(\mathscr{C})$. For $C^0$,
there exists an exact sequence
$$0\ra C^{-1}\ra P^0\ra C^0\ra 0$$ in $\mathscr{A}$ with $P^0\in \mathscr{P}(\mathscr{A})$ and $C^{-1}\in \mathscr{C}$.
Consider the following pullback diagram
$$
\xymatrix@R=20pt@C=20pt{ & 0 \ar@{-->}[d] & 0 \ar[d]& &\\
 & C^{-1}\ar@{==}[r] \ar@{-->}[d] &C^{-1} \ar[d]& &\\
0 \ar@{-->}[r] & {G}^0 \ar@{-->}[d] \ar@{-->}[r] & P^0 \ar[d] \ar@{-->}[r] &G^1 \ar@{==}[d] \ar@{-->}[r] & 0\\
0 \ar[r] & G \ar[r] \ar[d] & C^0 \ar[r] \ar[d] & G^1 \ar[r] & 0 &\\
 & 0 & 0. & & }
$$
By the above argument, we have $\Ext^1_{\mathscr{A}}({G}^0,L)\cong \Ext^2_{\mathscr{A}}(G^1,L)=0$. Because
the leftmost column splits by Proposition \ref{2.7}(2), $G$ is isomorphic to a direct summand of $G^0$
and $\Ext^1_{\mathscr{A}}({G},L)=0$, which shows that $L\in \mathcal{G}(\mathscr{C})^{\perp_1}$.


(3)  It is trivial that  $\mathscr{I}(\mathscr{A})\subseteq\mathcal{G}(\mathscr{C})^{\perp_1}$.
By Proposition \ref{2.7}, we have that $\mathcal{G}(\mathscr{C})^{\perp_1}$ is closed under
cokernels of monomorphisms. Thus $\mathcal{G}(\mathscr{C})^{\perp_1}$ is  injectively coresolving.
\end{proof}

Before giving some applications of Theorem \ref{3.3}(2), consider the following example.

\begin{example}\label{3.4}
{\rm Let $Q$ be a quiver:
$$
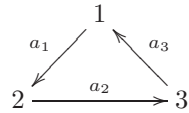
\xymatrix@R=20pt@C=20pt{
& 1\ar[ld]_{a_1}&\\
2\ar[rr]^{a_2}&&3\ar[lu]_{a_3}
}
$$
and $I=<a_1a_3a_2,a_2a_1a_3>$. Let $R=kQ/I$ with $k$ a field. Then the Auslander-Reiten quiver
$\Gamma(\mod R)$ of $\mod R$ is as follows.
$$
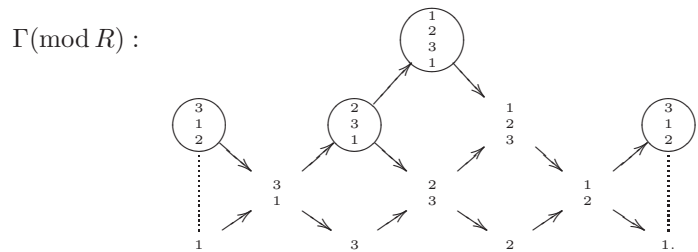
\Gamma(\mod R): \ \ \
\xymatrix@R=5pt@C=10pt{&&&*+[o][F]{\tiny\begin{array}{c}
      {1} \\ {2} \\ {3} \\ {1}
\end{array}}\ar[dr]&&&\\
*+[o][F]{\tiny\begin{array}{c}
      3 \\ 1 \\ 2
\end{array}}\ar[dr]\ar@{.}[dd]&&*+[o][F]{\tiny\begin{array}{c}
      2 \\ 3 \\ 1
\end{array}}\ar[dr]\ar[ur]&&{\tiny\begin{array}{c}  1 \\ 2
\\ 3  \end{array}}\ar[dr]&&*+[o][F]{\tiny\begin{array}{c}
      3 \\ 1 \\ 2
\end{array}}\ar@{.}[dd]\\
&{\tiny\begin{array}{c} 3 \\ 1\end{array}}\ar[ur]\ar[dr]&&{\tiny\begin{array}{c}  2 \\
3 \end{array}}\ar[ur]\ar[dr]&&{\tiny\begin{array}{c}    1 \\ 2 \ar[ur]\ar[dr] \end{array}}&\\
{\tiny \begin{array}{c}1\end{array}}\ar[ur]&&{\tiny \begin{array}{c}3\end{array}}
\ar[ur]&&{\tiny \begin{array}{c}2\end{array}}\ar[ur]&&{\tiny \begin{array}{c}1.\end{array}}
}
$$
By a direct computation, we have
$$
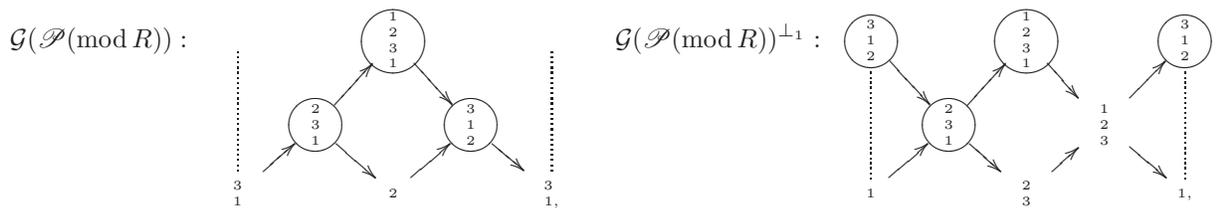
\mathcal{G}(\mathscr{P}(\mod R)):\ \
\xymatrix@R=5pt@C=10pt{&&*+[o][F]{\tiny\begin{array}{c}
      {1} \\ {2} \\ {3} \\ {1}
\end{array}}\ar[dr]&&\\
&*+[o][F]{\tiny\begin{array}{c}
      2 \\ 3 \\ 1
\end{array}}\ar[dr]\ar[ur]&&*+[o][F]{\tiny\begin{array}{c}
      3 \\ 1 \\ 2
\end{array}}\ar[dr]\\
{\tiny\begin{array}{c} 3 \\ 1\end{array}}\ar[ur]\ar@{.}[uu]&&
{\tiny \begin{array}{c}2\end{array}}
\ar[ur]&&{\tiny\begin{array}{c} 3\  \\ 1,\ar@{.}[uu]\end{array}}}
\ \ \ \
\mathcal{G}(\mathscr{P}(\mod R))^{\bot_1}: \ \ \xymatrix@R=5pt@C=10pt{
*+[o][F]{\tiny\begin{array}{c}
      3 \\ 1 \\ 2
\end{array}}\ar[dr]&& *+[o][F]{\tiny\begin{array}{c}
      {1} \\ {2} \\ {3} \\ {1}
\end{array}}\ar[dr]&&*+[o][F]{\tiny\begin{array}{c}
      3 \\ 1 \\ 2
\end{array}}\\
&*+[o][F]{\tiny\begin{array}{c}
      2 \\ 3 \\ 1
\end{array}}\ar[dr]\ar[ur]&&{\tiny\begin{array}{c}  1 \\ 2
\\ 3  \end{array}}\ar[dr]\ar[ur]&\\
{\tiny \begin{array}{c}1\ar@{.}[uu]\end{array}}\ar[ur]&&{\tiny\begin{array}{c} 2 \\
3 \end{array}}\ar[ur]&&{\tiny \begin{array}{c}1,\ar@{.}[uu]\end{array}}
}
$$
where the terms marked by  circles are indecomposable projective modules in $\mod R$.
Then we have $\mathcal{G}(\mathscr{P}(\mod R))\cap\mathcal{G}(\mathscr{P}(\mod R))^{\bot_1}=\mathscr{P}(\mod R)$.}
\end{example}

In general, we have the following

\begin{corollary}\label{3.5}
If $\mathscr{C}$ is closed under direct summands, then for any $n\geq 0$, we have
$$\mathcal{G}(\mathscr{C})^{\leq n}\cap\mathcal{G}(\mathscr{C})^{\perp_1}=\mathscr{C}^{\leq n}.$$
\end{corollary}

\begin{proof}
By Example \ref{3.1}(1), we have $\mathscr{C}^{\leq n}\subseteq\mathcal{G}(\mathscr{C})^{\leq n}\cap\mathcal{G}(\mathscr{C})^{\perp_1}$.

Now let $M\in \mathcal{G}(\mathscr{C})^{\leq n}\cap\mathcal{G}(\mathscr{C})^{\perp_1}$. By \cite[Theorem 5.8]{Hu},
there exists an exact sequence
$$0\ra K_n\ra C_{n-1}\ra \cdots\ra C_0\ra M\ra 0$$ in $\mathscr{A}$ with all $C_i$ in $\mathscr{C}$ and $K_n\in\mathcal{G}(\mathscr{C})$.
By Theorem \ref{3.3}(2), we have $K_n\in\mathcal{G}(\mathscr{C})^{\perp_1}$. Because $\mathscr{C}$ is closed under direct summands
by assumption,
it follows easily from the definition of $\mathcal{G}(\mathscr{C})$ that $K_n\in \mathscr{C}$ and $M\in \mathscr{C}^{\leq n}$.
\end{proof}

\begin{proposition}\label{3.6}
For any $M\in \mathscr{A}$, the following statements are equivalent.
\begin{enumerate}
\item[(1)] $M\in \mathcal{G}(\mathscr{C})^{\perp_1}$.
\item[(2)]  The functor $\Hom_\mathscr{A}(-,M)$ is exact with respect to any short exact sequence in $\mathscr{A}$
ending with an object in $\mathcal{G}(\mathscr{C})$.
\item[(3)] Every short exact sequence starting with $M$ is $\Hom_\mathscr{A}(\mathcal{G}(\mathscr{C}),-)$-exact.
\end{enumerate}
If, moreover, $R$ is a commutative ring, $\mathscr{A}=\Mod R$ and $\mathscr{C}=\mathscr{P}(\Mod R)$,
then the above conditions are equivalent to the following
\begin{enumerate}
\item[(4)] $\Hom_R(Q,M)\in\mathcal{G}(\mathscr{P}(\Mod R))^{\perp_1}$ for any  $Q\in \mathscr{P}(\Mod R)$.
\end{enumerate}
\end{proposition}

\begin{proof}
$(1)\Leftrightarrow (2) \Leftrightarrow (3)$  It is easy.

Now let $R$ be a commutative ring.

$(1)\Rightarrow (4)$ For any $G\in \mathcal{G}(\mathscr{P}(\Mod R))$, we have an exact sequence
$$0\ra K\buildrel {f} \over \ra P\ra G\ra 0\eqno{(3.1)}$$ in $\Mod R$ with $P\in\mathscr{P}(\Mod R)$.
Let $Q\in\mathscr{P}(\Mod R)$. Then
$$0\ra Q\otimes_RK\buildrel {1_Q\otimes f} \over\longrightarrow Q\otimes_RP\ra Q\otimes_RG\ra 0$$ is exact.
It is easy to check that $Q\otimes_R G\in\mathcal{G}(\mathscr{P}(\Mod R))$. Then
$\Ext^1_R(Q\otimes_RG, M)=0$ by (1), and so $\Hom_R(1_Q\otimes f, M)$ is epic.
By the adjoint isomorphism, we have that $\Hom_R(f,\Hom_R(Q,M))$ is also epic. So
applying the functor $\Hom_R(-,\Hom_R(Q,M))$ to (3.1) we get $\Ext^1_R(G,\Hom_R(Q,M))=0$,
and hence $\Hom_R(Q,M)\in \mathcal{G}(\mathscr{P}(\Mod R))^{\perp_1}$.

$(4)\Rightarrow (1)$ It is trivial by setting $Q=R$.
\end{proof}

In the following result, we characterize categories over which all objects are in $\mathcal{G}(\mathscr{C})^{\perp_1}$.

\begin{proposition}\label{3.7} Assume that $\mathscr{C}$ is closed under direct summands.
Consider the following conditions.
\begin{enumerate}
\item[(1)] $\mathcal{G}(\mathscr{C})^{\perp_1}=\mathscr{A}$.
\item[(2)] $\mathcal{G}(\mathscr{C})\subseteq\mathcal{G}(\mathscr{C})^{\perp_1}$.
\item[(3)] $\mathcal{G}(\mathscr{C})=\mathscr{C}$.
\end{enumerate}
Then we have $(1)\Rightarrow (2)\Rightarrow (3)$.
If $\mathscr{C}$ is a projective generator for $\mathscr{A}$, then all of them are equivalent.
\end{proposition}

\begin{proof}
The implication $(1)\Rightarrow (2)$ is trivial.

$(2)\Rightarrow (3)$ Let $G\in \mathcal{G}(\mathscr{C})$. Then there exists an exact sequence
$$0\ra G_1\ra C_0\ra G\ra 0$$ in $\mathscr{A}$ with $C_0\in\mathscr{C}$ and $G_1\in\mathcal{G}(\mathscr{C})$.
By (2), we have that $G_1\in\mathcal{G}(\mathscr{C})^{\perp_1}$ and the above exact sequence splits.
Thus as a direct summand of $C_0$, $G\in \mathscr{C}$ by assumption.

If $\mathscr{C}$ is a projective generator for $\mathscr{A}$, then the implication $(3)\Rightarrow (1)$ follows directly.
\end{proof}

Let $\mathscr{X}$ be a subcategory of $\mod R$ containing $\mathscr{P}(\mod R)$. We use $\underline{\mathscr{X}}$ to
denote the stable category of $\mathscr{X}$ modulo $\mathscr{P}(\mod R)$.
We end this section by giving two examples about $\mathcal{G}(\mathscr{P}(\mod R))^{\perp_1}$.

\begin{example}\label{3.8}
{ Let $Q_1$ and $Q_2$ be the following two quivers
$$
Q_1: \ \xymatrix@R=8pt@C=8pt{1\ar@/^/[rr]^{\alpha_1}&&2\ar@/^/[ll]^{\alpha_2}\ar[rr]^{\alpha_3}&&3\ar@(dr,ur)[]^{\alpha_4}}
\ \ \ \ \ \ \ \ \ \
Q_2: \  \xymatrix@R=8pt@C=8pt{a\ar@/^/[rr]^{\alpha_a}&&b\ar@/^/[ll]^{\alpha_b}\ar[rr]^{\alpha_c}&&c&&d,\ar[ll]_{\alpha_d}}
$$
and let $I_1=<\alpha_2\alpha_1,\alpha_1\alpha_2,\alpha_4\alpha_3,\alpha_4^2>$ and
$I_2=<\alpha_b\alpha_a,\alpha_a\alpha_b>$. Let $R_1=KQ_1/I_1$ and $R_2=KQ_2/I_2$.
Note that $R_2$ is Gorenstein and $R_1$ is not Gorenstein.
The Auslander-Reiten quivers of $\mod R_1$ and $\mod R_2$ are as follows.
$$\Gamma(\operatorname{mod} R_1):
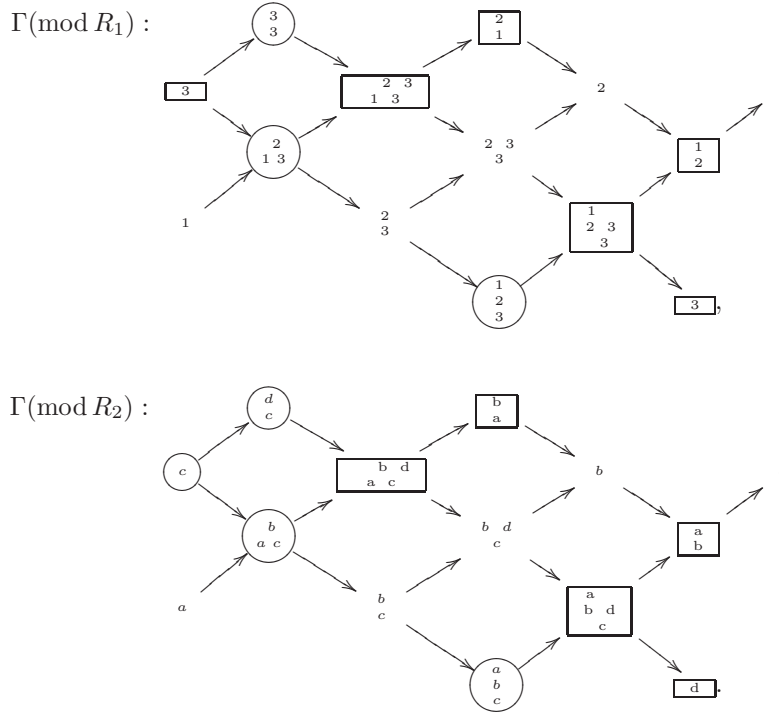
\xymatrix@R=5pt@C=10pt{
&*+[o][F]{\tiny\begin{array}{c}
      {3} \\ {3}
\end{array}}\ar[rd]&&{\tiny\begin{tabular}{|c|}
  \hline
      {2} \\ {1} \\
  \hline
\end{tabular}}\ar[rd]&&&\\
{\tiny\begin{tabular}{|c|}
  \hline
      {3} \\
  \hline
\end{tabular}}\ar[ur]\ar[dr]&&
{\tiny\begin{tabular}{|c|}
  \hline
   \ \ \  \ {2} \ 3 \\ 1 \ {3} \\
  \hline
\end{tabular}}\ar[ur]\ar[dr]&&
{\tiny\begin{array}{c}  2 \end{array}}\ar[dr]&&
{\tiny\begin{array}{c}  1 \end{array}}\\
&*+[o][F]{\tiny\begin{array}{c}
     \ {2}  \\ 1 \ {3} 
\end{array}}\ar[ur]\ar[dr]&&
{\tiny\begin{array}{c}
      {2} \ \ 3 \\  {3}
\end{array}}\ar[ur]\ar[dr]&&
{\tiny\begin{tabular}{|c|}
  \hline
      {1} \\ {2} \\
  \hline
\end{tabular}}\ar[ur]&\\
{\tiny\begin{array}{c}  1 \end{array}}\ar[ur]&&
{\tiny\begin{array}{c}
      {2} \\  {3}
\end{array}}\ar[ur]\ar[dr]&&
{\tiny\begin{tabular}{|c|}
  \hline
  1 \ \ \  \\
      {2} \ 3  \\  \ {3} \\
  \hline
\end{tabular}}\ar[ur]\ar[dr]&&\\
&&&*+[o][F]{\tiny\begin{array}{c}
     1\\ {2} \\ {3} 
\end{array}}\ar[ur]&&
{\tiny\begin{tabular}{|c|}
  \hline
      {3} \\
  \hline
\end{tabular},}
}
$$

$$\Gamma(\operatorname{mod} R_2):
\xymatrix@R=5pt@C=10pt{
&*+[o][F]{\tiny\begin{array}{c}
      {d} \\ {c} 
\end{array}}\ar[rd]&&{\tiny\begin{tabular}{|c|}
  \hline
      {b} \\ {a} \\
  \hline
\end{tabular}}\ar[rd]&&&\\
*+[o][F]{\tiny\begin{array}{c}
      {c} 
\end{array}}\ar[ur]\ar[dr]&&
{\tiny\begin{tabular}{|c|}
  \hline
   \ \ \  \ {b} \ d \\ a \ {c} \\
  \hline
\end{tabular}}\ar[ur]\ar[dr]&&
{\tiny\begin{array}{c}  b \end{array}}\ar[dr]&&
{\tiny\begin{array}{c}  a \end{array}}\\
&*+[o][F]{\tiny\begin{array}{c}
     \ {b}  \\ a \ {c} 
\end{array}}\ar[ur]\ar[dr]&&
{\tiny\begin{array}{c}
      {b} \ \ d \\  {c}
\end{array}}\ar[ur]\ar[dr]&&
{\tiny\begin{tabular}{|c|}
  \hline
      {a} \\ {b} \\
  \hline
\end{tabular}}\ar[ur]&\\
{\tiny\begin{array}{c}  a \end{array}}\ar[ur]&&
{\tiny\begin{array}{c}
      {b} \\  {c}
\end{array}}\ar[ur]\ar[dr]&&
{\tiny\begin{tabular}{|c|}
  \hline
  a \ \ \  \\
      {b} \ d  \\  \ {c} \\
  \hline
\end{tabular}}\ar[ur]\ar[dr]&&\\
&&&*+[o][F]{\tiny\begin{array}{c}
     a\\ {b} \\ {c} 
\end{array}}\ar[ur]&&
{\tiny\begin{tabular}{|c|}
  \hline
      {d} \\
  \hline
\end{tabular}.}
}
$$
Then we have
\begin{enumerate}
\item[(1)] The objects marked in a cycle or a box are indecomposable objects in $\mathcal{G}(\mathscr{P}(\mod R_i))^{\perp_1}$ ($i=1,2$);
in particular, the objects marked in a cycle are indecomposable objects in $\mathscr{P}(\mod R_i)$ ($i=1,2$).
\item[(2)] $\underline{\mod} R_1\simeq \underline{\mod} R_2$ and
$\frac{\operatorname{mod} R_1}{{\mathcal{G}(\mathscr{P}(\mod R_1))^{\perp_1}}}\simeq \frac{{\operatorname{mod} R_2}}{{\mathcal{G}(\mathscr{P}(\mod R_2))^{\perp_1}}}$.
\item[(3)]
${\mathcal{G}(\mathscr{P}(\mod R_1))^{\perp_1}}\underline{\nsim} \ {\mathcal{G}(\mathscr{P}(\mod R_2))^{\perp_1}}$ and
$\underline{{\mathcal{G}(\mathscr{P}(\mod R_1))^{\perp_1}}}\simeq \ \underline{{\mathcal{G}(\mathscr{P}(\mod R_2))^{\perp_1}}}$.
\end{enumerate}}
\end{example}

\begin{example}\label{3.9}
{ Let $Q_1$ and $Q_2$ be the following two quivers
$$
Q_1: \ \xymatrix@R=10pt@C=10pt{
1\ar[rr]^{\alpha_1}&& 2\ar[rr]^{\alpha_2} && 3\ar[ld]^{\alpha_3}\\
&&&4\ar[lu]^{\alpha_4}&}
\ \ \ \ \ \
Q_2: \ \xymatrix@R=10pt@C=10pt{a && b \ar[ll]_{\alpha_a}\ar[rr]^{\alpha_b} && c,\ar[ld]^{\alpha_c}\\
&&&d\ar[lu]^{\alpha_d}&}
$$
and let $I_1=<\alpha_3\alpha_2,\alpha_4\alpha_3,\alpha_2\alpha_4>$ and
$I_2=<\alpha_c\alpha_b,\alpha_d\alpha_c,\alpha_b\alpha_d>$. Let $R_1=KQ_1/I_1$ and $R_2=KQ_2/I_2$.
Then the Auslander-Reiten quivers of $\mod R_1$ and $\mod R_2$ are as follows.
$$
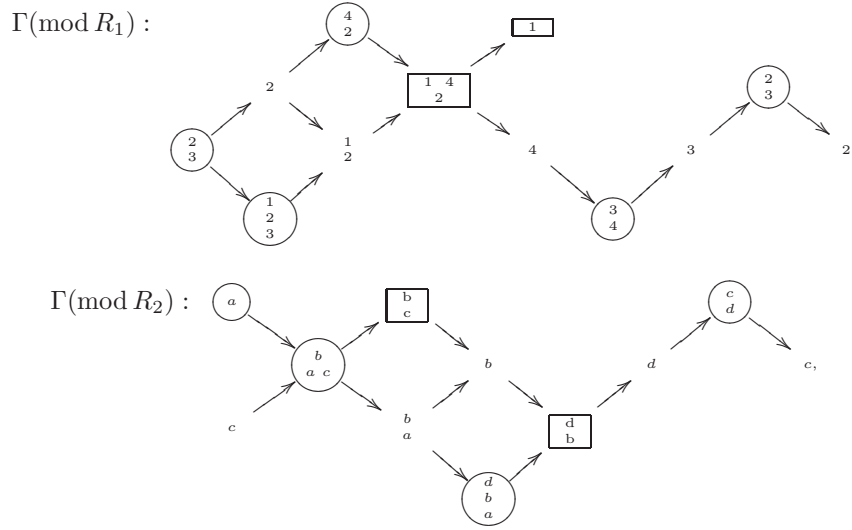
\Gamma(\operatorname{mod} R_1): \  \xymatrix@R=5pt@C=10pt{
&&*+[o][F]{\tiny\begin{array}{c}
      {4} \\ {2} 
\end{array}}\ar[rd]&&{\tiny\begin{tabular}{|c|}
  \hline
      {1} \\
  \hline
\end{tabular}}&&&&\\
&{\tiny\begin{array}{c}  2 \end{array}}\ar[ru]\ar[rd]&&{\tiny\begin{tabular}{|c|}
  \hline
      {1} \ 4 \\ {2} \\
  \hline
\end{tabular}}\ar[ru]\ar[rd]&&&&*+[o][F]{\tiny\begin{array}{c}
       {2} \\ {3} 
\end{array}}\ar[rd]&\\
*+[o][F]{\tiny\begin{array}{c}
       {2} \\ {3} 
\end{array}}\ar[ru]\ar[rd]&&{\tiny\begin{array}{c}    1 \\
2 \end{array}}\ar[ru]&&{\tiny\begin{array}{c}    4 \end{array}}\ar[rd]&&
{\tiny\begin{array}{c} 3 \end{array}}\ar[ru]&&{\tiny\begin{array}{c}  2, \end{array}}\\
&*+[o][F]{\tiny\begin{array}{c}
      {1} \\ {2} \\ {3} 
\end{array}}\ar[ru]&&&&*+[o][F]{\tiny\begin{array}{c}
      {3} \\ {4} 
\end{array}}\ar[ru]&&&
}
$$
$$
\Gamma(\operatorname{mod} R_2): \ \xymatrix@R=5pt@C=10pt{
*+[o][F]{\tiny\begin{array}{c}
      {a} 
\end{array}}\ar[rd]&&{\tiny\begin{tabular}{|c|}
  \hline
      {b} \\ {c} \\
  \hline
\end{tabular}}\ar[rd]&&&&*+[o][F]{\tiny\begin{array}{c}
       {c} \\ {d} 
\end{array}}\ar[rd]&\\
&*+[o][F]{\tiny\begin{array}{c}
       {b} \\ a \ {c} 
\end{array}}\ar[ru]\ar[rd]&&{\tiny\begin{array}{c} b\end{array}}
\ar[rd]&&{\tiny\begin{array}{c} d\end{array}}\ar[ru]&&{\tiny\begin{array}{c}
c,\end{array}}\\
{\tiny\begin{array}{c}  c\end{array}}\ar[ru]&&{\tiny\begin{array}{c}  b\\
a\end{array}}\ar[ru]\ar[rd]&&{\tiny\begin{tabular}{|c|}
  \hline
      {d} \\ {b} \\
  \hline
\end{tabular}}\ar[ru]&&&\\
&&&*+[o][F]{\tiny\begin{array}{c}
      {d} \\ {b} \\  {a} 
\end{array}}\ar[ru]&&&&\\
}
$$
Then we have
\begin{enumerate}
\item[(1)] The objects marked in a cycle or a box are indecomposable objects in $\mathcal{G}(\mathscr{P}(\mod R_i))^{\perp_1}$ ($i=1,2$);
in particular, the objects marked in a cycle are indecomposable objects in $\mathscr{P}(\mod R_i)$ ($i=1,2$).
\item[(2)] $\underline{\mod} R_1 \underline{\nsim} \ {\underline{\mod} R_2}$ and
$\frac{\operatorname{mod} R_1}{{\mathcal{G}(\mathscr{P}(\mod R_1))^{\perp_1}}}\simeq \frac{{\operatorname{mod} R_2}}{{\mathcal{G}(\mathscr{P}(\mod R_2))^{\perp_1}}}$.
\item[(3)]
${\mathcal{G}(\mathscr{P}(\mod R_1))^{\perp_1}}\underline{\nsim} \ {\mathcal{G}(\mathscr{P}(\mod R_2))^{\perp_1}}$ and
$\underline{{\mathcal{G}(\mathscr{P}(\mod R_1))^{\perp_1}}}\simeq \ \underline{{\mathcal{G}(\mathscr{P}(\mod R_2))^{\perp_1}}}$.
\end{enumerate} }
\end{example}

\section{The special precovered category of $\mathcal{G}(\mathscr{C})$}

In this section, we introduce and investigate the special precovered category of $\mathcal{G}(\mathscr{C})$
in terms of the properties of $\mathcal{G}(\mathscr{C})^{\perp_1}$.

\begin{proposition}\label{4.1}
\begin{enumerate}
\item[]
\item[(1)] Let $M\in\mathcal{G}(\mathscr{C})^{\perp_1}$ and $f: C\twoheadrightarrow M$ be an epimorphism in $\mathscr{A}$ with $C\in\mathscr{C}$.
Then $\operatorname{Ker}f\in\mathcal{G}(\mathscr{C})^{\perp_1}$ and $f$ is a special $\mathcal{G}(\mathscr{C})$-precover of $M$.
\item[(2)] 
Consider an exact sequence
$$0\ra M'\ra C\ra M\ra 0.\eqno{(4.1)}$$
If $M'$ admits special $\mathcal{G}(\mathscr{C})$-precover, then so is $M$. The converse is true if $\mathscr{C}$ is a generator
for $\mathcal{G}(\mathscr{C})^{\perp_1}$ and (4.1) is $\Hom_{\mathscr{A}}(\mathscr{C},-)$-exact.
\end{enumerate}
\end{proposition}

\begin{proof}
(1) The assertion follows from Example \ref{3.1}(1) and Theorem \ref{3.3}(2).

(2) Assume that $M'$ admits a special $\mathcal{G}(\mathscr{C})$-precover and
$$0\ra N\ra G\ra M'\ra 0$$ is an exact sequence in $\mathscr{A}$ with $G\in\mathcal{G}(\mathscr{C})$
and $N\in \mathcal{G}(\mathscr{C})^{\perp_1}$.
Combining it with the following $\Hom_{\mathscr{A}}(-,\mathscr{C})$-exact exact sequence
$$0\ra G\buildrel {i} \over\ra C^0\buildrel {p} \over\ra G^1\ra 0$$
in $\mathscr{A}$ with $C^0\in\mathscr{C}$ and $G^1\in\mathcal{G}(\mathscr{C})$,
we get the following commutative diagram with exact columns and rows
$$\xymatrix@R=20pt@C=20pt{
  &0\ar[d]&&&\\
  &N\ar[d]&&&\\
0 \ar[r] & G \ar[d] \ar[r]^i & C^0 \ar@{-->}[d]^g \ar[r]^p &G^1\ar@{-->}[d]^h \ar[r] & 0\\
0 \ar[r] &M' \ar[r] \ar[d]& C \ar[r]  & M \ar[r] & 0 &\\
&0.&&&
}$$
Adding the exact sequence $\xymatrix@C=0.5cm{
0 \ar[r] & 0 \ar[r]^{} & C \ar[r]^{1_C} & C \ar[r] & 0 }$ to the middle row,
we obtain the following commutative diagram with exact columns and rows
$$\xymatrix@R=20pt@C=20pt{
  &0\ar[d]&&&\\
  &N\ar[d]&&&\\
0 \ar[r] & G \ar[d] \ar[r]^{{i\choose 0}} & C^0\oplus C \ar[d]^{(g, 1_C)}
\ar[r]^{{p \ 0\  \choose  \ 0 \ 1_C}} &G^1\oplus C\ar@{-->}[d]^{h'} \ar[r] & 0\\
0 \ar[r] &M' \ar[r] \ar[d]& C \ar[r]\ar[d]  & M \ar[r] & 0 &\\
&0&0,&&
}$$
which can be completed to a commutative diagram with exact columns and rows as follows.
$$\xymatrix@R=20pt@C=20pt{
  &0\ar[d]&0\ar@{-->}[d]&0\ar@{-->}[d]&\\
 0\ar@{-->}[r] &N\ar@{-->}[r]\ar[d]&C'\ar@{-->}[r]\ar@{-->}[d]&M''\ar@{-->}[r]\ar@{-->}[d]&0\\
0 \ar[r] & G \ar[d] \ar[r]^{{i\choose 0}} & C^0\oplus C \ar[d]^{(g,1_C)}
\ar[r]^{{p \ 0\  \choose  \ 0 \ 1_C}} &G^1\oplus C\ar[d]^{h'} \ar[r] & 0\\
0 \ar[r] &M' \ar[r] \ar[d]& C \ar[r]\ar[d]  & M \ar[r] \ar[d] & 0 &\\
&0&0&0.&
}$$
Note that $G^1\oplus C\in \mathcal{G}(\mathscr{C})$. Moreover, since $N\in \mathcal{G}(\mathscr{C})^{\perp_1}$,
we have $M''\in \mathcal{G}(\mathscr{C})^{\perp_1}$ by Theorem \ref{3.3}(3).
Thus the rightmost column in the above diagram is a special $\mathcal{G}(\mathscr{C})$-precover of $M$.

Now let $\mathscr{C}$ be a generator for $\mathcal{G}(\mathscr{C})^{\perp_1}$ and (4.1) be
$\Hom_{\mathscr{A}}(\mathscr{C},-)$-exact. Assume that $M$ admits a special $\mathcal{G}(\mathscr{C})$-precover and
$$0\ra L\ra G\ra M\ra 0,$$
$$0\ra L'\ra C'\ra L\ra 0$$
are exact sequences in $\mathscr{A}$ with $G\in \mathcal{G}(\mathscr{C})$, $L\in \mathcal{G}(\mathscr{C})^{\perp_1}$ and $C'\in\mathscr{C}$.
By \cite[Lemma 3.1(1)]{Hu}, we get the following commutative diagram with exact columns and rows
$$
\xymatrix@=0.6cm{
&0\ar[d]&0\ar@{-->}[d]&0\ar[d]&\\
0\ar@{-->}[r]&L'\ar@{-->}[r]\ar[d]&G'\ar@{-->}[r]\ar@{-->}[d]&M'\ar@{-->}[r]\ar[d]&0\\
0\ar@{-->}[r]&C'\ar@{-->}[r]\ar[d]
&C'\oplus C\ar@{-->}[r]\ar@{-->}[d]&C\ar@{-->}[r]\ar[d]&0\\
0\ar[r]&L\ar[r]\ar[d]&G\ar[r]\ar@{-->}[d]&M\ar[r]\ar[d]&0\\
&0&0&0.&
 }
 $$
By Proposition \ref{2.7}(2) and Theorem \ref{3.3}(2), we have $L'\in \mathcal{G}(\mathscr{C})^{\perp_1}$ and the leftmost column
is $\Hom_{\mathscr{A}}(\mathscr{C},-)$-exact. So the middle column is also $\Hom_{\mathscr{A}}(\mathscr{C},-)$-exact.
On the other hand,
the middle column is $\Hom_{\mathscr{A}}(-,\mathscr{C})$-exact by Proposition \ref{2.7}(2).
So $G'\in \mathcal{G}(\mathscr{C})$ by \cite[Proposition 4.7(5)]{Hu}, and hence
the upper row is a special $\mathcal{G}(\mathscr{C})$-precover of $M'$.
\end{proof}

We introduce the following

\begin{definition}\label{4.2}
We call $\spc(\mathcal{G}(\mathscr{C})):=\{A\in\mathscr{A}\mid A$
admits a special $\mathcal{G}(\mathscr{C})$-precover$\}$ the {\it special precovered category} of $\mathcal{G}(\mathscr{C})$.
\end{definition}

It is trivial that $\spc(\mathcal{G}(\mathscr{C}))$ is the largest subcategory of $\mathscr{A}$ such that $\mathcal{G}(\mathscr{C})$
is special precovering in it. In particular, $\spc(\mathcal{G}(\mathscr{C}))=\mathscr{A}$ if and only if $\mathcal{G}(\mathscr{C})$
is special precovering in $\mathscr{A}$. For the sake of convenience, we say that a subcategory $\mathscr{X}$
of $\mathscr{A}$ is {\it closed under $\mathscr{C}$-stable direct summands} provided that the condition $X\oplus C\in\mathscr{X}$
with $C\in \mathscr{C}$ implies $X\in\mathscr{X}$.


\begin{theorem}\label{4.3}
\begin{enumerate}
\item[]
\item[(1)] $\spc(\mathcal{G}(\mathscr{C}))$ is closed under  extensions.
\item[(2)] $\spc(\mathcal{G}(\mathscr{C}))$ is closed under $\mathscr{C}$-stable direct summands.
\end{enumerate}
\end{theorem}

\begin{proof}

(1) Let
$$0\ra L\ra M\ra N\ra 0$$ be an exact sequence in $\mathscr{A}$.
 Assume that $L$ and $N$ admit special $\mathcal{G}(\mathscr{C})$-precovers and
$$0\ra L'\ra G_L\buildrel {f}\over\ra L\ra 0,$$
$$0\ra N'\ra G_N\buildrel {g}\over\ra N\ra 0$$ are exact sequences in $\mathscr{A}$
with $G_L,G_N\in \mathcal{G}(\mathscr{C})$ and $L',N'\in \mathcal{G}(\mathscr{C})^{\perp_1}$. Consider the following pullback diagram
$$\xymatrix@R=20pt@C=20pt{
0 \ar@{-->}[r] & L \ar@{==}[d] \ar@{-->}[r] & Q \ar@{-->}[d]^{\alpha} \ar@{-->}[r] &G_N\ar[d]^{g} \ar@{-->}[r] & 0\\
0 \ar[r] &L \ar[r] & M \ar[r]  & N \ar[r] & 0. &
}$$
Since $\Ext^2_R(G_N,L')=0$ by Proposition \ref{2.7}(1), we get an epimorphism
$\Ext^1_R(G_N,f):\Ext^1_R(G_N,G_L)\ra \Ext^1_R(G_N,L)$. It induces the following commutative diagram with exact rows
$$\xymatrix@R=20pt@C=20pt{
0 \ar@{-->}[r] & G_L \ar[d]^f \ar@{-->}[r] & G_M \ar@{-->}[d]^{\beta} \ar@{-->}[r] &G_N\ar@{==}[d] \ar@{-->}[r] & 0\\
0 \ar[r] & L \ar@{=}[d] \ar[r] & Q \ar[d]^{\alpha} \ar[r] &G_N\ar[d]^{g} \ar[r] & 0\\
0 \ar[r] &L \ar[r] & M \ar[r]  & N \ar[r] & 0. &
}$$
Set $M':=\Ker\alpha\beta$. Then we get the following commutative diagram with exact columns and rows
$$\xymatrix@R=20pt@C=20pt{
&0\ar[d]&0\ar[d]&0\ar[d]&\\
0\ar@{-->}[r] &L'\ar@{-->}[r]\ar[d]&M'\ar@{-->}[r]\ar[d]&N'\ar@{-->}[r]\ar[d]&0\\
0 \ar[r] & G_L \ar[d] \ar[r] & G_M \ar[d] \ar[r] &G_N\ar[d] \ar[r] & 0\\
0 \ar[r] &L \ar[r] \ar[d]& M \ar[r]\ar[d]  & N \ar[r] \ar[d] & 0 &\\
&0&0&0.&
}$$
Note that $G_M\in\mathcal{G}(\mathscr{C})$ (by \cite[Corollary 4.5]{SSW08}) and
$M'\in \mathcal{G}(\mathscr{C})^{\perp_1}$ (by Theorem \ref{3.3}(1)).
Thus the middle column in the above diagram is a special $\mathcal{G}(\mathscr{C})$-precover of $M$.
This proves that $\spc(\mathcal{G}(\mathscr{C}))$ is closed under extensions.

(2) Let $M\in \spc(\mathcal{G}(\mathscr{C}))$ and
$$0\ra K\ra G\ra M\ra 0$$ be an exact sequence in $\mathscr{A}$
with $G\in \mathcal{G}(\mathscr{C})$ and $K\in \mathcal{G}(\mathscr{C})^{\perp_1}$.
Assume that $M\cong L\oplus C$ with $C\in\mathscr{C}$,
we have an exact and split sequence
$$0\ra C\ra M\ra L\ra 0$$ in $\mathscr{A}$. Consider the following pullback diagram
$$\xymatrix@R=20pt@C=20pt{& & 0 \ar@{-->}[d] & 0 \ar[d]& &\\
0 \ar@{-->}[r] & K \ar@{==}[d] \ar@{-->}[r] & L' \ar@{-->}[d] \ar@{-->}[r] &C \ar[d] \ar@{-->}[r] & 0\\
0 \ar[r] & K \ar[r] & G \ar[r] \ar@{-->}[d] & M \ar[d] \ar[r] & 0 &\\
& & L \ar@{==}[r] \ar@{-->}[d] & L \ar[d]& &\\
& & 0 & 0. & & }$$
Since $K,C\in \mathcal{G}(\mathscr{C})^{\perp_1}$, we have $L'\in\mathcal{G}(\mathscr{C})^{\perp_1}$ by Theorem \ref{3.3}(1).
Thus the middle column in the above diagram is a special $\mathcal{G}(\mathscr{C})$-precover of $L$.
\end{proof}

The following question seems to be interesting.

\begin{question} \label{4.4}
Is $\spc(\mathcal{G}(\mathscr{C}))$ closed under direct summands?
\end{question}

The following result shows that $\spc(\mathcal{G}(\mathscr{C}))$ possesses certain minimality, which generalizes \cite[Theorem 6.8(1)]{Ta}.

\begin{theorem} \label{4.5}
Assume that
$\mathscr{C}$ is a generator for $\mathcal{G}(\mathscr{C})^{\perp_1}$. Then we have
\begin{enumerate}
\item[(1)]
$\mathcal{G}(\mathscr{C})^{\perp_1}\cup \mathcal{G}(\mathscr{C})\subseteq \spc(\mathcal{G}(\mathscr{C}))$ and $\spc(\mathcal{G}(\mathscr{C}))$
is closed under extensions and $\mathscr{C}$-stable direct summands.
\item[(2)] $\spc(\mathcal{G}(\mathscr{C}))$ is the minimal subcategory with respect to the property (1) as above.
\end{enumerate}
\end{theorem}

To prove this theorem, we need the following

\begin{lemma}\label{4.6}
Let $$0\ra K\ra G\ra M\ra 0$$ be an exact sequence in $\mathscr{A}$ with $K\in\mathcal{G}(\mathscr{C})^{\perp_1}$ and $G\in\mathcal{G}(\mathscr{C})$.
Then there exists an exact sequence
$$0\ra G\ra M\oplus C\ra K'\ra 0$$
in $\mathscr{A}$ with $K'\in \mathcal{G}(\mathscr{C})^{\perp_1}$ and $C\in\mathscr{C}$.
\end{lemma}

\begin{proof}
Let $$0\ra K\ra G\ra M\ra 0$$ be an exact sequence in $\mathscr{A}$ with $K\in\mathcal{G}(\mathscr{C})^{\perp_1}$ and $G\in\mathcal{G}(\mathscr{C})$.
Since $G\in\mathcal{G}(\mathscr{C})$, there exists a $\Hom_{\mathscr{A}}(\mathscr{C},-)$-exact exact sequence
$$0\ra G\ra C\ra G'\ra 0$$ in $\mathscr{A}$ with $C\in\mathscr{C}$ and $G'\in\mathcal{G}(\mathscr{C})$. Consider the following pushout diagram
$$
 \xymatrix@R=20pt@C=20pt{& & 0 \ar[d] & 0 \ar@{-->}[d]& &\\
0 \ar[r] & K \ar@{==}[d] \ar[r] & G \ar[d] \ar[r] &M \ar@{-->}[d] \ar[r] & 0\\
0 \ar@{-->}[r] & K \ar@{-->}[r] & C \ar@{-->}[r] \ar[d] & K' \ar@{-->}[d] \ar@{-->}[r] & 0 &\\
& & G' \ar@{==}[r] \ar[d] & G' \ar@{-->}[d]& &\\
& & 0 & 0. & & }
$$
Since $K,C\in\mathcal{G}(\mathscr{C})^{\perp_1}$, we have $K'\in \mathcal{G}(\mathscr{C})^{\perp_1}$ by Theorem \ref{3.3}(3).

Consider the following pullback diagram
$$
\xymatrix@R=20pt@C=20pt{& & 0 \ar@{-->}[d] & 0 \ar[d]& &\\
& & G \ar@{==}[r] \ar@{-->}[d] & G \ar[d]& &\\
0 \ar@{-->}[r] & M \ar@{==}[d] \ar@{-->}[r] & Q \ar@{-->}[d] \ar@{-->}[r] &C \ar[d] \ar@{-->}[r] & 0\\
0 \ar[r] & M \ar[r] & K' \ar[r] \ar@{-->}[d] & G' \ar[d] \ar[r] & 0 &\\
& & 0 & 0. & & }
$$
Since the middle column in the first diagram is $\Hom_{\mathscr{A}}(\mathscr{C},-)$-exact,
so is the rightmost column in this diagram. Then the middle row in the second diagram is also
$\Hom_{\mathscr{A}}(\mathscr{C},-)$-exact by \cite[Lemma 2.4(1)]{Hu}, and in particular, it splits.
Thus $Q\cong M\oplus C$ and the middle column in the second diagram is the desired exact sequence.
\end{proof}

{\it Proof of Theorem \ref{4.5}.} (1) It follows from Proposition \ref{4.1}(1) and Theorem \ref{4.3}.

(2) Let $\mathscr{X}$ be a subcategory of $\mathscr{A}$ such that $\mathcal{G}(\mathscr{C})^{\perp_1}\cup \mathcal{G}(\mathscr{C})
\subseteq \mathscr{X}$ and $\mathscr{X}$ is closed under extensions and $\mathscr{C}$-stable direct summands.
Let $M\in \spc(\mathcal{G}(\mathscr{C}))$. Then by Lemma \ref{4.6}, we have an exact sequence
$$0\ra G\ra M\oplus C\ra K'\ra 0$$ in $\mathscr{A}$ with $K'\in \mathcal{G}(\mathscr{C})^{\perp_1}$,
$G\in \mathcal{G}(\mathscr{C})$ and $C\in\mathscr{C}$. Because $G,K'\in \mathscr{X}$,
we have that $M\oplus C\in \mathscr{X}$ and $M\in \mathscr{X}$. It follows that $\spc(\mathcal{G}(\mathscr{C}))\subseteq \mathscr{X}$.

\vspace{0.2cm}

As an immediate consequence of Theorem \ref{4.5}, we get the following

\begin{corollary}\label{4.7}
Assume that $\mathcal{G}(\mathscr{P}(\Mod R))$ is special precovering in $\Mod R$ and $\mathscr{X}$ is a subcategory of $\Mod R$.
If $\mathcal{G}(\mathscr{P}(\Mod R))^{\perp_1}\cup \mathcal{G}(\mathscr{P}(\Mod R))\subseteq \mathscr{X}$
and $\mathscr{X}$ is closed under extensions and $\mathscr{P}(\Mod R)$-stable direct summands, then $\mathscr{X}=\Mod R$.
\end{corollary}

\begin{proof}
By assumption, we have $\spc(\mathcal{G}(\mathscr{P}(\Mod R)))=\Mod R$.
Now the assertion follows from Theorem \ref{4.5}.
\end{proof}

We collect some known classes of rings $R$ satisfying that
$\mathcal{G}(\mathscr{P}(\Mod R))$ is special precovering in $\Mod R$ as follows.

\begin{example}\label{4.8}
{\rm For any one of the following rings $R$, $\mathcal{G}(\mathscr{P}(\Mod R))$ is special precovering in $\Mod R$.
\begin{enumerate}
\item[(1)] Commutative Noetherian rings of finite Krull dimension (\cite[Remark 5.8]{CFH11}).
\item[(2)] Rings in which all projective left $R$-modules have finite injective dimension (\cite[Corollary 4.3]{WL});
especially, Gorenstein rings (that is, $n$-Gorenstein rings for some $n\geq 0$).
\item[(3)] Right coherent rings in which all flat $R$-modules
have finite projective dimension (\cite[Theorem 3.5]{ADH} and \cite[Proposition 8.10]{BGM});
especially, right coherent and left perfect rings, and right Artinian rings.
\end{enumerate}}
\end{example}

We recall the following definition from \cite{Hu14}.

\begin{definition}\label{4.9}
{ Let $\mathscr{C}$, $\mathscr{T}$ and $\mathscr{E}$ be
subcategories of $\mathscr{A}$ with $\mathscr{C}\subseteq\mathscr{T}$.
\begin{enumerate}
\item[(1)] $\mathscr{C}$ is called an {\it $\mathscr{E}$-proper generator}
(resp. {\it $\mathscr{E}$-coproper cogenerator}) for $\mathscr{T}$
if for any object $T$ in $\mathscr{T}$, there exists a
$\Hom_{\mathscr{A}}(\mathscr{E},-)$ (resp.
$\Hom_{\mathscr{A}}(-,\mathscr{E})$)-exact exact sequence $0\to
T'\to C \to T \to 0$ (resp. $0\to T\to C \to T' \to 0$) in
$\mathscr{A}$ such that $C$ is an object in $\mathscr{C}$ and
$T'$ is an object in $\mathscr{T}$.

\item[(2)] $\mathscr{T}$ is called {\it
$\mathscr{E}$-preresolving} in $\mathscr{A}$ if the following
conditions are satisfied.

(i) $\mathscr{T}$ admits an $\mathscr{E}$-proper generator.

(ii) $\mathscr{T}$ is {\it closed under $\mathscr{E}$-proper
extensions}, that is, for any
$\Hom_{\mathscr{A}}(\mathscr{E},-)$-exact exact sequence $0\to
A_1\to A_2 \to A_3 \to 0$ in $\mathscr{A}$, if both $A_1$ and $A_3$
are objects in $\mathscr{T}$, then $A_2$ is also an object in
$\mathscr{T}$.

An $\mathscr{E}$-preresolving subcategory $\mathscr{T}$ of
$\mathscr{A}$ is called {\it $\mathscr{E}$-resolving} if the
following condition is satisfied.

(iii) $\mathscr{T}$ is {\it closed under kernels of
$\mathscr{E}$-proper epimorphisms}, that is, for any
$\Hom_{\mathscr{A}}(\mathscr{E},-)$-exact exact sequence $0\to
A_1\to A_2 \to A_3 \to 0$ in $\mathscr{A}$, if both $A_2$ and $A_3$
are objects in $\mathscr{T}$, then $A_1$ is also an object in
$\mathscr{T}$.
\end{enumerate}}
\end{definition}

In the following, we investigate when $\spc(\mathcal{G}(\mathscr{C}))$ is $\mathscr{C}$-resolving. We need the following two lemmas.

\begin{lemma}\label{4.10}
For any $M\in \spc(\mathcal{G}(\mathscr{C}))$, there exists a $\Hom_{\mathscr{A}}(\mathscr{C},-)$-exact exact sequence
$$0\ra K\ra C\ra M\ra 0$$  in $\mathscr{A}$ with $C\in\mathscr{C}$.
\end{lemma}

\begin{proof}
Let $M\in \spc(\mathcal{G}(\mathscr{C}))$. Then there exists a $\Hom_{\mathscr{A}}(\mathscr{C},-)$-exact exact sequence
$$0\ra K'\ra G\ra M\ra 0$$ in $\mathscr{A}$ with $G\in \mathcal{G}(\mathscr{C})$ and $K'\in \mathcal{G}(\mathscr{C})^{\perp_1}$.
For $G$, there exists a $\Hom_{\mathscr{A}}(\mathscr{C},-)$-exact exact sequence
$$0\ra G'\ra C\ra G\ra 0$$ in $\mathscr{A}$ with $C\in \mathscr{C}$ and $G'\in \mathcal{G}(\mathscr{C})$.
Consider the following pullback diagram
$$
\xymatrix@R=20pt@C=20pt{ & 0 \ar@{-->}[d] & 0 \ar[d]& &\\
 & G'\ar@{==}[r] \ar@{-->}[d] &G' \ar[d]& &\\
0 \ar@{-->}[r] & K \ar@{-->}[d] \ar@{-->}[r] & C \ar[d] \ar@{-->}[r] &M \ar@{==}[d] \ar@{-->}[r] & 0\\
0 \ar[r] &K' \ar[r] \ar@{-->}[d] & G \ar[r] \ar[d] & M \ar[r] & 0 &\\
 & 0 & 0. & & }
$$
By \cite[Lemma 2.5]{Hu}, the middle row is $\Hom_{\mathscr{A}}(\mathscr{C},-)$-exact, as desired.
\end{proof}

\begin{lemma}\label{4.11} Assume that
$\mathscr{C}$ is a generator for $\mathcal{G}(\mathscr{C})^{\perp_1}$.
Given a $\Hom_{\mathscr{A}}(\mathscr{C},-)$-exact exact sequence
$$0\ra L\ra M\ra N\ra 0$$ in $\mathscr{A}$, we have
\begin{enumerate}
\item[(1)] If $M,N\in \spc(\mathcal{G}(\mathscr{C}))$, then $L\in \spc(\mathcal{G}(\mathscr{C}))$.
\item[(2)] If $L,M\in \spc(\mathcal{G}(\mathscr{C}))$ and there exists a $\Hom_{\mathscr{A}}(\mathscr{C},-)$-exact
exact sequence $$0\ra K\ra C\ra N\ra 0$$ in $\mathscr{A}$ with $C\in\mathscr{C}$, then $N\in \spc(\mathcal{G}(\mathscr{C}))$.
\end{enumerate}
\end{lemma}

\begin{proof}
Let $  0\ra L\ra M\ra N\ra 0$
be a $\Hom_{\mathscr{A}}(\mathscr{C},-)$-exact exact sequence in $\mathscr{A}$.

(1) Assume that $M,N\in \spc(\mathcal{G}(\mathscr{C}))$. By Lemma \ref{4.10},
there exists a $\Hom_{\mathscr{A}}(\mathscr{C},-)$-exact exact sequence
$$0\ra K\ra C\ra N\ra 0$$ in $\mathscr{A}$ with $C\in\mathscr{C}$.
Consider the following pullback diagram
$$\xymatrix@R=20pt@C=20pt{& & 0 \ar@{-->}[d] & 0 \ar[d]& &\\
& & K \ar@{==}[r] \ar@{-->}[d] & K \ar[d]& &\\
0 \ar@{-->}[r] & L \ar@{==}[d] \ar@{-->}[r] & T \ar@{-->}[d] \ar@{-->}[r] &C \ar[d] \ar@{-->}[r] & 0\\
0 \ar[r] & L \ar[r] & M \ar[r] \ar@{-->}[d] & N \ar[d] \ar[r] & 0 &\\
& & 0 & 0. & & }
$$
By Proposition \ref{4.1}(2), $K\in \spc(\mathcal{G}(\mathscr{C}))$. Then it follows from Theorem \ref{4.3}(1) and the exactness of
the middle column that $T\in \spc(\mathcal{G}(\mathscr{C}))$. Notice that the middle row is $\Hom_{\mathscr{A}}(\mathscr{C},-)$-exact
by \cite[Lemma 2.4(1)]{Hu}, so it splits and $T\cong L\oplus C$. Thus $L\in \spc(\mathcal{G}(\mathscr{C}))$ by Theorem \ref{4.3}(2).

(2) Assume $L,M\in \spc(\mathcal{G}(\mathscr{C}))$ and there exists a $\Hom_{\mathscr{A}}(\mathscr{C},-)$-exact
exact sequence $$0\ra K\ra C\ra N\ra 0$$ in $\mathscr{A}$ with $C\in\mathscr{C}$.
As in the above diagram, since $L,C\in  \spc(\mathcal{G}(\mathscr{C}))$, we have $T\in \spc(\mathcal{G}(\mathscr{C}))$ by
Theorem \ref{4.3}(1). Moreover, the middle column is $\Hom_{\mathscr{A}}(\mathscr{C},-)$-exact by \cite[Lemma 2.4(1)]{Hu}.
So $K\in \spc(\mathcal{G}(\mathscr{C}))$ by (1), and hence $N\in \spc(\mathcal{G}(\mathscr{C}))$ by Proposition \ref{4.1}(2).
\end{proof}

Now we are ready to prove the following

\begin{theorem}\label{4.12}
If $\mathscr{C}$ is a generator for $\mathcal{G}(\mathscr{C})^{\perp_1}$, then
$\spc(\mathcal{G}(\mathscr{C}))$ is $\mathscr{C}$-resolving in $\mathscr{A}$ with a $\mathscr{C}$-proper generator $\mathscr{C}$.
\end{theorem}

\begin{proof}
Following Theorem \ref{4.3}(1) and Lemma \ref{4.11}(1), we know that $\spc(\mathcal{G}(\mathscr{C}))$ is closed under
$\mathscr{C}$-proper extensions and kernels of $\mathscr{C}$-proper epimorphisms. Now let $M\in \spc(\mathcal{G}(\mathscr{C}))$.
Then by Lemma \ref{4.10}, there exists a $\Hom_{\mathscr{A}}(\mathscr{C},-)$-exact exact sequence
$$0\ra K\ra C\ra M\ra 0$$
in $\mathscr{A}$ with $C\in\mathscr{C}$. By Proposition \ref{4.1}(2), we have $K\in \spc(\mathcal{G}(\mathscr{C}))$.
It follows that $\mathscr{C}$ is a $\mathscr{C}$-proper generator for $\spc(\mathcal{G}(\mathscr{C}))$ and
$\spc(\mathcal{G}(\mathscr{C}))$ is a $\mathscr{C}$-resolving.
\end{proof}

As a consequence, we get the following

\begin{corollary}\label{4.13}
If $\mathscr{C}$ is a projective generator for $\mathscr{A}$, then $\spc(\mathcal{G}(\mathscr{C}))$ is projectively resolving
and injectively coresolving in $\mathscr{A}$.
\end{corollary}

\begin{proof}
Let $\mathscr{C}$ be a projective generator for $\mathscr{A}$. Because $\mathcal{G}(\mathscr{C})^{\perp_1}$ is projectively resolving
by Theorem \ref{3.3}(2), $\mathscr{C}$ is also a projective generator for $\mathcal{G}(\mathscr{C})^{\perp_1}$.
It follows from Theorem \ref{4.12} that $\spc(\mathcal{G}(\mathscr{C}))$ is projectively resolving.
Now let $I$ be an injective object in $\mathscr{A}$ and
$$0\ra K \ra P \buildrel {f}\over\ra I \ra 0$$ an exact sequence in $\mathscr{A}$ with $P\in\mathscr{C}$.
Then it is easy to see that $K\in\mathcal{G}(\mathscr{C})^{\perp_1}$ by Example \ref{3.1}(1) and Theorem \ref{3.3}(2).
So $f$ is a special $\mathcal{G}(\mathscr{C})$-precover of $I$ and $I\in \spc(\mathcal{G}(\mathscr{C}))$.
On the other hand, by Lemma \ref{4.11}(2), we have that $\spc(\mathcal{G}(\mathscr{C}))$ is closed under
cokernels of monomorphisms. Thus we conclude that $\spc(\mathcal{G}(\mathscr{C}))$ is injectively coresolving.
\end{proof}

The following corollary is an immediate consequence of Corollary \ref{4.13}, in which the second assertion generalizes
\cite[Theorem 6.8(2)]{Ta}.

\begin{corollary}\label{4.14}
\begin{enumerate}
\item[]
\item[(1)] $\spc(\mathcal{G}(\mathscr{P}(\Mod R)))$ is projectively resolving and injectively coresolving in $\Mod R$.
\item[(2)] If $R$ is a left Noetherian ring, then $\spc(\mathcal{G}(\mathscr{P}(\mod R)))$ is projectively resolving and
injectively coresolving in $\mod R$.
\end{enumerate}
\end{corollary}

Let $\spe(\mathcal{G}(\mathscr{C}))$ be the subcategory of $\mathscr{A}$ consisting of objects
admitting special $\mathcal{G}(\mathscr{C})$-preenvelopes. We point out that
the dual versions on ${^{\bot_1}\mathcal{G}(\mathscr{C})}$ and $\spe(\mathcal{G}(\mathscr{C}))$
of all of the above results also hold true by using completely dual arguments.

\vspace{0.5cm}

{\bf Acknowledgements.}
This research was partially supported by NSFC (Grant No. 11571164),
a Project Funded by the Priority Academic Program Development of Jiangsu Higher Education Institutions,
the University Postgraduate Research and Innovation Project of Jiangsu Province 2016 (No. KYZZ16\_0034),
Nanjing University Innovation and Creative Program for PhD candidate (No. 2016011).
The authors thank the referees for the useful suggestions.

\end{document}